\documentclass[10pt]{amsart}
\usepackage{amsaddr}
\usepackage{amssymb,  mathrsfs}
\usepackage{tikz-cd} 
\newtheorem{thm}{Theorem}
\newtheorem{prop}[thm]{Proposition}
\theoremstyle{definition}
\newtheorem{defi}[thm]{Definition}
\newtheorem{lem}[thm]{Lemma}

\newtheorem{rem}[thm]{Remark}
\newtheorem{cor}[thm]{Corollary}

\newcommand{\m}{{}^{-1}}

\newcommand{\Z}{\mathbb{Z}}
\newcommand{\K}{\mathbb{K}}

\newcommand{\F}{\mathbb{F}}

\usepackage[all]{xy}
\begin{document}

\author{Daniel Gon\c{c}alves}
\address{ Departamento de Matem\'atica, Universidade Federal de Santa Catarina, 88040-900\\
	Florian\'opolis, Brasil}
\email{daemig@gmail.com}

\author{Laura  Orozco}
\address{Escuela de Matematicas, Universidad Industrial de Santander, Cra. 27 Calle 9  UIS
	Edificio 45, Bucaramanga, Colombia} \email{lnorogar@correo.uis.edu.co}
\author{H\'ector  Pinedo}

\address{Escuela de Matematicas, Universidad Industrial de Santander, Cra. 27 Calle 9  UIS
	Edificio 45, Bucaramanga, Colombia} \email{hpinedot@uis.edu.co}




\date{\today}
\title[The graded structure of LPA viewed as partial skew group rings]{The graded structure of Leavitt path algebras viewed as partial skew group rings}

\maketitle

\begin{abstract} Let $E$ be a directed graph, $\K$ be a field, and $\F$ be the free group on the edges of $E$. In this work, we use the isomorphism between Leavitt path algebras and partial skew group rings to endow $L_\K(E)$ with an $\F$-grading and study some algebraic properties of this grading. More precisely, we show that graded cleanness, graded unit regularity, and strong grading of $L_\K(E)$ are all equivalent. 

\end{abstract}

\section{Introduction} The Leavitt path algebra associated with a directed graph $E$ and a field $\K$ is a ring  $L_\K(E)$ which is an algebraic analog of the graph Cuntz–Krieger
C*-algebra. Leavitt path algebras were introduced in 
 \cite{AMP} and  simultaneously in \cite{AAP}. Since their introduction, Leavitt path algebras
have gained importance, and several of their algebraic properties have been studied in terms of 
properties of the underlying graph (see for instance \cite{A, AAS, AALP, AP}). Each Leavitt path algebra may be viewed as a $\Z$-graded ring (see Section~4 in \cite{NO}) and a systematic study of 
properties of the $\Z$-grading is considered in \cite{H, strong, NO, NO2,  vavs2021}.  Moreover, in \cite{PMS} the authors study  Leavitt path algebras endowed with canonical $G$-gradings that contain strongly  $G$-graded rings as direct summands, where $G$ is an arbitrary group.

On the other hand, partial actions of groups and partial skew group rings are introduced and studied in \cite{DE} as algebraic analogs of $C^*$-partial crossed products and have become a useful tool in several branches of mathematics, \cite{D}. The connections between partial actions and Leavitt path algebras, in the context of actions of groups, groupoids, or inverse semigroups, are made in \cite{GR}, \cite{Yoneda}, and \cite{BeuterCordeiro}, respectively. We will focus on the connection between partial actions of groups and Leavitt path algebras, as described in \cite{GR}, where the Leavitt path algebra associated with a graph is realized as a partial skew group ring. From this realization, a new grading is obtained on Leavitt path algebras, namely a grading over the free group $\F$ on the edges of the graph. Moreover,
new proofs of the Cuntz–Krieger uniqueness theorem and the simplicity criteria
for Leavitt Path Algebras are obtained via partial crossed product theory as well as a characterization of Artinian Leavitt Path Algebras, \cite{NYOP}. 

The interaction between Leavitt path algebras and partial skew group rings is extended to include Leavitt path algebras associated with ultragraphs (which are generalizations of graphs) in \cite{GR1}. In this context, which includes Leavitt path algebras of graphs, 
the skew group ring point of view is used to prove the reduction theorem, \cite{GR2}, and to describe graphs for which the grading (both by $\mathbb Z$ and $\F$) is strong or epsilon-strong, see \cite{GR3}. Nevertheless, the development of the properties of the  $\F$-grading on a Leavitt path algebra is still in its infancy.

In this work, we study some structural properties of the $\F$-grading on a Leavitt path algebra $L_\K(E)$, such as graded cleanness and graded unit-regularity (see Proposition~\ref{loop1} and Proposition~\ref{loop2}, respectively) and show that these conditions are equivalent to $L_\K(E)$ being strongly graded, see Theorem~\ref{loop3}. In the last section of this work, given any graph, we modify the construction in \cite{GR} (including some formal series) and present an alternative partial skew group ring that is isomorphic to the Leavitt path algebra associated with the graph.

\section{Preliminaries} In this section, we present some notions and results on graded rings and partial actions which will be useful in the work. 
\subsection{Graded rings}
Let $R$ be a ring, which we always consider associative, and $G$ be a group with identity $e$. We say that $R$ is {\it graded} by $G$, or $G$-{\it graded}, or  {\it graded} when the group $G$ is clear from the context, if $R = \bigoplus_{g \in G} R_g$, where $R_g$ is an additive subgroup of $R$ for each $g \in G$, and
$R_g R_h \subseteq R_{gh}$ for any $g,h\in G$ (where for $X,Y$ subsets of $R$ the product $XY$ is defined as the set of finite sums of elements $xy,$ with $x\in X$ and $y\in Y$). If the equality $R_g R_h=R_{gh}$ is true then we say that $R$ is {\it strongly graded}.  It is not difficult to see that for a graded ring $R,$ the set  $R_e$ is a subring of $R$ and $R_gR_{g^{-1}}$ is an ideal of $R_e,$ for any $g\in G.$ Moreover, the non-zero elements in $ \bigcup_{g\in G} R_g$ are called homogeneous.

 An element $x$ in a unital ring $R$ is \textit{clean} if $x = u + a$ for some unit $u$ and some idempotent $a$. If every element of $R$ is clean then we say that $R$ is clean. In the case of graded rings, we have, from \cite{vavs2021}, the following definition.

\begin{defi} 
 A graded unital ring $R$ is \textit{graded clean} if every homogeneous element $x \in R$ has a graded clean decomposition, that is, every $x\in R$ can be written as a sum of a homogeneous unit and a homogeneous idempotent. 
\end{defi}

The following result, proved in \cite[Lemma 2.2]{vavs2021}, gives us a characterization of graded clean rings.  We include the proof here for completeness. 
\begin{lem}\label{cleang}
Let $R$ be a $G$-graded unital ring. Then $R$ is graded clean if, and only if, $R_{e}$ is clean and each nonzero element of $R_{g}$ is invertible for every $g\neq e$.
\end{lem}
\begin{proof}
Suppose that $R$ is graded clean.  We prove that $R_{e}$ is clean: let $x\in R_{e}$. Since $R$ is graded clean, there is a homogeneous invertible element $u\in R$ and a homogeneous idempotent $a$ such that $x=u+a$. So, we only need to prove that $u,a\in R_{e}$, which is true since  $R_{e}$ is a direct summand of $R$. Now, we prove that any element of $R_{g}$ is invertible for all $g\neq e$. Let $x\in R_{g}$. Since $R$ is graded clean, we have that $x=u+a$ for a homogeneous invertible element $u$ and a homogeneous idempotent $a$. Since $a$ is a homogeneous idempotent, we have that $a\in R_{e}$, in which case $a=0$ and $x=u $ is invertible.

Conversely, suppose that $R_{e}$ is clean and each element of $R_{g}$ is invertible for all $g\neq e$. Let $x\in R_{g}$. If $g\neq e$ then $x$ is invertible and $x=x+0$. If $g=e$ then, since $R_{e}$ is clean, we can write $x=u+a$ with $u$ invertible, $a$ idempotent and $u,a\in R_{e}$. We conclude that $R$ is graded clean.
\end{proof}

A unital ring $R$ is \textit{unit-regular} if, for any $x \in R$, there is an invertible $u \in R$ such that $x = xux$. The notion of unital regularity in the realm of graded rings is defined as follows (see \cite{vavs2021}).

\begin{defi} We say that a unital graded ring $R$ is \textit{graded unit-regular} if each homogeneous element $x\in R$ can be written as $x=xux$, where $u$ is an invertible and homogeneous element.

\end{defi}
\begin{lem}
 Let $R$ be a graded unit-regular ring. Then, every nonzero component $R_{g}$ contains an invertible element.
\end{lem}
\begin{proof}
Let $0\neq x\in R_{g}$. Then, there exists a homogeneous and invertible element $u$  such that $x=xux$, which implies that $u\in R_{g^{-1}}$ and its inverse $u^{-1}$ belongs to $ R_{g}$.
\end{proof}

\subsection{Partial actions} A {\it  partial action} of the group $G$ on a ring $R$ is a pair
$\alpha = ( D_g  , \alpha_g )_{g \in G}$,
where for each $g \in G$, $D_g$ is an ideal of $R,$
$\alpha_g : D_{g^{-1}} \rightarrow D_g$ is an isomorphism
of rings, and for each $(g,h) \in G \times G$,

\begin{itemize}

\item[(P1)] $\alpha_e = {\rm id}_R$;

\item[(P2)] $\alpha_g(D_{g^{-1}}\cap D_h) = D_g\cap D_{gh}$;

\item[(P3)] if $r \in D_{h^{-1}}\cap D_{(gh)^{-1}}$, 
then $\alpha_g ( \alpha_h (r) ) =
 \alpha_{gh}(r).$
\end{itemize}
Moreover, we say that $\alpha$ is {\it global} if $D_g=R,$ for any $g\in G.$
\vspace{0.2cm}

Associated with the notion of partial action, we have the generalization of the concept of skew group ring.

\begin{defi}\label{skewr}
	Given a partial action $\alpha=( D_g  , \alpha_g )_{g \in G}$ on a ring $R,$ the {\it  partial skew group ring} $R \star_{\alpha} G$
	is the direct sum $\bigoplus_{g \in G} D_g \delta_g$,
	in which the $\delta_g$'s are formal symbols. Addition is defined point-wise and multiplication is defined by the bi-additive 
	extension of the relations
\begin{equation}\label{prodc}(r \delta_g) (r' \delta_h) =\alpha_g(\alpha_{g^{-1}}(r) r' )  \delta_{gh},\end{equation}
for $g,h \in G$, $r \in D_g$ and $r' \in D_h$. 	
\end{defi}

 In general, the ring $R \star_{\alpha} G$ 
 is not associative, as shown in \cite[Example 3.5]{DE}. Nevertheless, it follows by \cite[Corollary 3.2]{DE}  that $R \star_{\alpha} G$ is associative provided that  $D_g$ is idempotent or non degenerate, for any $g\in G.$  Moreover, by equation \eqref{prodc} follows that $R \star_{\alpha} G$ is a $G$-graded ring, where $(R \star_{\alpha} G)_g=D_g\delta_g,$ for all $g\in G.$


The following is straightforward. 

\begin{prop} \label{globs} Let $\alpha=( D_g  , \alpha_g )_{g \in G}$ be a partial action such that $D_g$ is idempotent for any $g\in G$. Then 
  $R\star_{\alpha} G$ 
  is strongly $G-$graded if, and only if, $\alpha$ is global.  
\end{prop}



\section{Leavitt path algebras as partial crossed products}\label{LPASR}

A directed graph $E=(E^0,E^1,r,s)$ consists of two countable sets $E^0$, $E^1$ and maps $r,s : E^1 \to E^0$ (called the range and source maps, respectively). The elements of $E^0$ are called \emph{vertices} and the elements of $E^1$ are called \emph{edges}. 
A vertex $v$ such that $s^{-1}(v)=\emptyset$ is called a \emph{sink}, such that $r^{-1}(v)=\emptyset$ is called a \emph{source} and such that $s^{-1}(v)=\emptyset=r^{-1}(v)$ is called \emph{isolated}. A vertex $v\in E^0$ such that $|s^{-1}(v)|=\infty$ is called an \emph{infinite emitter}.
A \emph{path} $\mu$ in 
$E$ is either a vertex or a sequence of edges $\mu = \mu_1 \ldots \mu_n$ such that $r(\mu_i)=s(\mu_{i+1})$ for $i\in \{1,\ldots,n-1\}$. For a path $\mu = \mu_1 \ldots \mu_n$, the range and source maps are extended as $s(\mu):=s(\mu_1)$ and $r(\mu)=r(\mu_n)$, respectively, and $n$ is the \emph{length} of $\mu$. We set the length of a vertex as zero.  A \emph{loop based on a vertex $v$} is an edge $e$ such that $s(e)=r(e)=v$. When the context is clear we will omit the base of the loop. A special role in this  work  is played by the graph $$
			
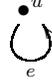
\xymatrix{ &\bullet^{u}\ar@(dl,dr)_{e} } 
$$
    which consists of one vertex and one loop (abusing the notation, we also say that such a graph is a loop).

Next, we recall the definition of the Leavitt path algebra associated with a graph (see \cite{AAP} for more details).

\begin{defi}\label{lpa}
Let $E$ be a directed graph and let $\K$ be a field.
The \emph{Leavitt path {$\K$}-algebra $L_{\K}(E)$ of $E$ with coefficients in $\K$} is the
$\K$-algebra generated by a set $\{v \mid v\in E^0\}$ of pairwise orthogonal idempotents, together with a set of elements $\{f \mid f\in E^1\} \cup
\{f^* \mid f\in E^1\}$, which satisfy the following  relations:
\begin{enumerate}
\item $s(f)f=fr(f)=f$, for all $f\in E^1$;
\item $r(f)f^*=f^*s(f)=f^*$, for all $f\in E^1$;
\item $f^*f'=\delta_{f,f'}r(f)$, for all $f,f'\in E^1$;
\item $v=\sum_{ \{ f\in E^1 \mid s(f)=v \} } ff^*$, for every $v\in E^0$ for which $s^{-1}(v)$ is non-empty and finite.
\end{enumerate}
\end{defi}

In \cite{GR}, it is shown that the Leavitt path algebra $L_{\K}(E)$ is isomorphic as a $\K$-algebra to a partial skew group ring $D(X) \star_\alpha \F$, where $D(X)$ is a  certain commutative $\K$-algebra and $\F$ is the free group generated by $E^1$.
For the reader's benefit, we briefly recall the construction of $D(X) \star_\alpha \F$.

Let $W$ denote the set of all finite paths of nonzero length in $E$, and let $W^{\infty}$ denote the set of all infinite paths in $E$.
Denote by $\F$ the free group generated by $E^1$.
We define a partial action of $\F$ on the set
\begin{displaymath}
	X=\{\xi\in W \mid r(\xi) \text{ is a sink}\}\cup \{v\in E^0 \mid v\text{ is a sink }\}\cup W^{\infty}.
\end{displaymath}
For each $g\in \F$, let $X_g$ be defined as follows:
\begin{itemize}
\item $X_e:=X$, where $e$ is the identity element of $\F$. 

\item $X_{b^{-1}}:=\{\xi\in X \mid s(\xi)=r(b)\}$, for all $b\in W$. 

\item $X_a:=\{\xi\in X \mid \xi_1\xi_2...\xi_{|a|}=a\}$, for all $a\in W$.

\item $X_{ab^{-1}}:=\{\xi\in X \mid \xi_1\xi_2...\xi_{|a|}=a\}=X_a$, for $ab^{-1}\in \F$ with $a,b\in W$, $r(a)=r(b)$ and $ab^{-1}$ in its reduced form.

\item $X_g:=\emptyset$, for all other $g \in \F$.

\end{itemize}

Let  $\theta_e:X_e\rightarrow X_e$ be the identity map. For $b\in W$, if $r(b)$ is not a sink, $\theta_b:X_{b^{-1}}\rightarrow X_b$ is defined by $\theta_b(\xi)=b\xi$ and $\theta_{b^{-1}}:X_b\rightarrow X_{b^{-1}}$ by
$\theta_{b^{-1}}(\eta)= \eta_{|b|+1}\eta_{|b|+2}\ldots$.
On the other hand, if $r(b)$ is a sink, then $X_{b^{-1}}=\{r(b)\}$ and $X_b=\{b\}$, in which case we define $\theta_b(r(b)) = b$ and $\theta_{b^{-1}}(b)=r(b)$.
Finally, for $a,b\in W$ with $r(a)=r(b)$ and $ab^{-1}$ in reduced form, $\theta_{ab^{-1}}:X_{ba^{-1}}\rightarrow X_{ab^{-1}}$ is defined by $\theta_{ab^{-1}}(\xi)=a\xi_{(|b|+1)}\xi_{(|b|+2)}\ldots$, with inverse 
$\theta_{ba^{-1}}:X_{ab^{-1}}\rightarrow X_{ba^{-1}}$ defined by $\theta_{ba^{-1}} (\eta)=b\eta_{(|a|+1)}\eta_{(|a|+2)}\ldots$ .

The pair $(X_g, \theta_g)_{g\in\F}$ is a partial action on the set level, which gives a partial action on a ring level  $(F(X_g), \alpha_g)_{g\in\F}$, where, for each $g\in \F$, $F(X_g)$ denotes the $\K$-algebra of all functions from $X_g$ to $\K$ (which is identified with the subalgebra of $F(X)$ given by the functions from $X$ to $\K$ that vanish outside $X_g$), and $\alpha_g:F(X_{g^{-1}})\rightarrow F(X_g)$ is defined by the rule $\alpha_g(f)=f\circ \theta_{g^{-1}}$.
We define another partial action in the following way:

For each $g\in \F$, and for each $v\in E^0$, define the characteristic maps $1_g:=\chi_{X_g}$ and $1_v:=\chi_{X_v}$, where $X_v=\{\xi\in X \mid s(\xi)=v\}$. Notice that $1_g$ is the identity element of $F(X_g)$ and  $1_a=1_{ab\m},$  for $ab^{-1}\in \F$ with $a,b\in W$, $r(a)=r(b)$ and $ab^{-1}$ in its reduced form. Finally, let
\begin{equation}\label{eqdx}
    D(X)= D_{e}=span_{\K}\{\{1_{g}\mid g\in \F\setminus\{0\}\}\cup\{1_{v}\mid v\in E^{0}\}\}\subseteq F(X),
\end{equation}
where $ span_{\K}$ means the $\K$-linear span, and for each $g\in \F\setminus\{e\}$, let $D_g\subseteq F(X_g)$ be defined as $1_g D_e$, that is,
\begin{displaymath}
	D_g= span_{\K}\{1_g1_h \mid h\in \F\}.
\end{displaymath}
By \cite[Lemma 2.4]{GR}, $D(X)$ is a $\K$-algebra and $D_g$, for $g\in \F$, is an ideal of $D(X)$.
From  \cite[Lemma 2.6]{GR} one has that \begin{equation}\label{alpha}
    \alpha_g(1_{g^{-1}}1_h)=1_g1_{gh}, \end{equation} for each $g\in \F$, this implies that 
one obtains that  $\alpha_g$ is a bijection from $D_{g^{-1}}$ onto $D_g$.
To simplify notation, this restriction map will also be denoted by $\alpha_g$.
Clearly, $\alpha_{g}:D_{g^{-1}}\rightarrow D_g$ is an isomorphism of {$\K$}-algebras and, furthermore, $(\alpha_g, D_g)_{g\in \F}$ is a partial action. By \cite[Proposition 3.2]{GR} the map 
\begin{equation}\label{isop}\varphi : L_\K(E) \to D(X) \star_\alpha \F\,\,, \varphi(f)=1_f \delta_f, \,\, \varphi(f^*)=1_{f^{-1}}\delta_{f^{-1}}\,\,\,\text{and}\,\,\, \varphi(v)=1_v \delta_e,\end{equation}  for any $f\in E^1$ and $v\in E^0$, is an isomorphism of {$\K$}-algebras.\medskip

Our goal is to show some results of the study of the $\mathbb{F}$-grading on Leavitt path algebras induced by the isomorphism defined above.

\begin{prop}\label{idemp}
The ideal $D_g$ defined above is idempotent for any $g\in\F$.
\end{prop}
\begin{proof}
	For each $g\neq e$ we have that $D_g$ is idempotent, since it is generated by the central idempotent $1_g$. It remains to show that $D_e$ is idempotent. For this, it is enough to prove that $D_{e}D_{e}\supseteq D_{e}$, since the other inclusion is clear. 
	
	Let $x\in D_{e}$.  Then, by \eqref{eqdx}, we have that
$$x=\sum_{i=1}^{n}\lambda_{i}1_{p_{i}q_{i}^{-1}}+\sum_{i=1}^{m}\gamma_{j}1_{v_{j}},$$
where $p_{i},q_{i}\in W$ with $r(p_i)=r(q_i)$  and $p_iq_i\m$ in its reduced form, for each $i\in \{1,2,\ldots, n\}$, $v_j\in E^0$ for all $j\in \{1,2,\ldots, m\}$, and $\lambda_{i},\gamma_{j}\in \K.$ For $g,h\in \F$ set $s(gh\m)=s(g)$ if $g\neq e,$ and $s(h\m)=r(h).$   Moreover, write $$V=\{s(p_{i}q_{i}^{-1})\mid1\leq  i\leq n\}\cup\{v_j\mid 1  \leq  i\leq m \},$$ 
$y=\sum_{v\in V}1_{v},$ and  let $\delta$ be the  Kronecker delta function. Then,
\begin{align*}
    xy&=\sum_{v\in V}\left(\sum_{i=1}^{n}\lambda_{i}1_{p_{i}q_{i}^{-1}}1_{v}+\sum_{i=1}^{m}\gamma_{j}1_{v_{j}}1_{v}\right)\\
    &=\sum_{v\in V}\left(\sum_{i=1}^{n}\lambda_{i}(\delta_{s(p_{i}q_{i}^{-1}),v}1_{p_{i}q_{i}^{-1}}1_{v})+\sum_{i=1}^{m}\gamma_{j}(\delta_{v_{j},v}1_{v_{j}}1_{v})\right)\\
    &=\sum_{i=1}^{n}\lambda_{i}1_{p_{i}q_{i}^{-1}}1_{s(p_{i}q_{i}^{-1})}+\sum_{i=1}^{m}\gamma_{j}1_{v_{j}}1_{v_{j}}\\
    &=\sum_{i=1}^{n}\lambda_{i}1_{p_{i}q_{i}^{-1}}+\sum_{i=1}^{m}\gamma_{j}1_{v_{j}}\\&= x,
\end{align*}
and $x\in D_{e}D_{e}$ as desired.
\end{proof}

As a consequence of Proposition \ref{idemp}, and the comment after Definition~\ref{skewr}, we obtain that the partial skew group ring $D(X)\star_{\alpha}\mathbb{F}$ is associative.  Below we describe when this partial skew ring is strongly graded (this result was proved, in the context of ultragraphs, in \cite[Theorem~6.3]{GR3}, but our proof is shorter and simpler).  

From now on, we consider only graphs with at least one edge, since otherwise, the free group on the edges is trivial.

\begin{prop}\label{loop}
Let $E$ be a graph with at least one edge. The ring $D(X)\star_{\alpha}\mathbb{F}$ is strongly $\mathbb{F}$-graded if, and only if, the graph $E$ is a loop.
\end{prop}
\begin{proof}
Suppose that $D(X)\star_{\alpha}\mathbb{F}$ is strongly graded. It follows by Proposition~\ref{globs} and Proposition~\ref{idemp} that $\alpha$ is global. If there exists $f,l\in E^{1}$ with $l\neq f$ then we have that  $X_{l}\neq X_f$, 
which implies that $1_{f}\neq 1_{l}$ and hence $D_{f}\neq D_{l}$, which contradicts that the action is global. So, $f=l$ and $E^{1}=\{f\}$. 
Let $v\in E^{0}$. Then, $1_{v}\in D_{e}=D_f$, and $1_{f}$ is the unit of $D_{e}$. Thus, $1_{v}=1_{f}1_{v}=1_{f}1_{s(f)}1_{v}$. 
Since the functions $1_{u}$, with $u\in E^0$, are orthogonal, we get that $v=s(f)$. Analogously, we obtain that $v=r(f)$. Thus $s(f)=r(f)$  and the graph is a loop. 
 
 Conversely, suppose that $E$ is a loop with  $E^{1}=\{f\}$. Then $X=W^\infty=\{f^\infty\}$, with $f^\infty$ the infinite path that arises by concatenating the edge $f$, and $\mathbb{F}=\langle f\rangle,$ so that every element  $p\in \mathbb{F}$ is of the form $p=f^n,$ for some $n\in \Z$. Hence,  $X=X_{f}=X_{p},$ and 
 consequently $1_{f}=1_{p}$ for each $p\in \mathbb{F}$. Thus, $D_{p}=D_{f}$ for each $p\in\mathbb{F}$ and the action is global, that is $D(X)\star_{\alpha}\mathbb{F}$ is strongly graded.
\end{proof}

From now on we assume that $L_{\K}(E)$ has the $\mathbb{F}$-grading given by the isomorphism \eqref{isop}. 
Next, we describe when $L_{\K}(E)$ is $\mathbb{F}$-graded clean and $\mathbb{F}$-graded unit-regular. 

\begin{prop}\label{loop1}
Let $E$ be a graph with at least one edge. The Leavitt path algebra $L_{\K}(E)$ is $\F$-graded clean if, and only if, $L_{\K}(E)$ is strongly $\mathbb F-$graded.
\end{prop}
\begin{proof}
Suppose that $L_{\K}(E)\cong D(X)\star_{\alpha}\mathbb{F}$ is $\mathbb{F}$-graded clean. Then it is unital and thus has a finite number of vertices. By Proposition~\ref{loop} we must verify that the graph is a loop. We first prove that the graph has only one vertex. Let $f\in E^{1}$ and suppose that there exists $v\neq s(f)$. Then, by Lemma \ref{cleang}, $0\neq 1_{f}\delta_{f}\in D_{f}\delta_{f}$ is invertible with inverse $w$ write $w=\varphi(u),$ where $\varphi$ is the ring isomorphism defined in  \eqref{isop}. Then $0\neq v =v(fu)=(vf)u=0,$
which is a contradiction. Therefore, $s(f)$ is the only vertex in $E$ and $1_{L_{\K}(E)}=s(f)$. Now, we prove that $f$ is the only edge of $E$. For this, we consider the set of idempotent elements
 $$S=\left\{\sum_{i=1}^{n}f_{i}f_{i}^{*}\in L_{\K}(E)\mid f_ {i}\in E^{1} \textit{ and }n\in \mathbb{N}\right\}$$
and define a partial order in $S$ by $p\leq q$ if, and only if, $pq=qp= p$. Take an edge $l\neq f$. Since $f^{*}f=r(f)=s(f)=1_{L_{\K}(E)}=fu$, we have that $f^{*}=f^{*} (f u)=(f^{*} f) u= u$. Thus, $ff^{*}=1_{L_{\K}(E)}$ and so $ff^{*}\geq ff^{*}+ll^{*}$. 
Clearly, $ff^{*}\leq ff^{*}+ll^{*} $. Hence $ff^{*}=ff^{*}+ll^{*}$, which implies that $ll^{*}=0$ and thus $l=lr(l)=ll^{* }l=0l=0$, which is a contradiction. It follows that $E$ is a loop and therefore $L_{\K}(E)$ is strongly graded.

Conversely, suppose that $D(X)\star_{\alpha}\mathbb{F}$ is strongly graded. By Proposition~\ref{loop}, the graph must be a loop with $E^{0}=\{v\}$ and $E^{1}=\{f\}$, in which case $X=\{f^\infty\}$  and $1_{v}\delta_{e}$ is the unit of
$D(X)\star_{\alpha}\mathbb{F}$. Since $1_{p}=1_{v}$ for all $p\in \mathbb{F}$, we have by  Equations \eqref{prodc} and  \eqref{alpha} that $(\lambda 1_{p}\delta_{p})(\alpha_{p\m}(\lambda\m1_{p})\delta_{p^{-1}})=1_{v}\delta_{e}$ for any $\lambda\in \K, \lambda\neq 0$,  that is, each  nonzero element of $D_{p}\delta_{p}$ has an  inverse and  $D_{e}\delta_{e}$ is clean. Therefore, by Lemma~\ref{cleang}, $D(X)\star_{\alpha}\mathbb{F}$ is graded clean.
\end{proof}

\begin{prop}\label{loop2} Suppose that $E$ has at least one edge. The Leavitt path algebra $L_{\K}(E)$ is $\F$-graded unit-regular if, and only if, $E$ is a loop.
\end{prop}
\begin{proof} Assume that the Leavitt path algebra $L_{\K}(E)$ is graded unit-regular, and suppose by contradiction that the graph is not a loop. This implies that either there are at least two distinct edges in $E^{1}$ or there is a single edge with a finite number of isolated vertices. We separate the proof accordingly to these two cases. 

{\bf Case 1:} Suppose that there are two distinct edges, say $l,f\in E^{1}$. 

Then, by hypothesis, there exists an invertible element $u\in D_{l^{-1}}\delta_{l^{-1}}$ such that $(1_{ l}\delta_{l})u(1_{l}\delta_{l})=1_{l}\delta_{l}$. Write
     $$u=\sum_{j=1}^{n}\lambda_j1_{l^{-1}}1_{q_{j}}\delta_{l^{-1}}\quad \text{ and  } \quad u^{-1 }=\sum_{i=1}^{m}\beta_i1_{l}1_{p_{i}}\delta_{l},$$
     with  $\lambda_j, \beta_i\in \K$ and $q_{j},p_{i}\in \mathbb{F}$  for all $i,j$. We know that $u^{-1}u(1_{f}\delta_{f})=1_{f}\delta_{f}$, but since $l\neq f$ we have by  equations \eqref{prodc} and  \eqref{alpha} that
    \begin{align*}
      u(1_{f}\delta_{f})&=\left(\sum_{j=1}^{n}\lambda_j1_{l^{-1}}1_{q_{j}} \delta_{l^{-1}}\right)(1_{f}\delta_{f})\\
      &= \alpha_{l^{-1}}\left(\alpha_{l}\left(\sum_{j=1}^{n}\lambda_j1_{l^ {-1}}1_{q_{j}}\right)1_{f}\right)\delta_{l^{-1}f}\\
      &=\sum_{j=1}^{n}\lambda_j\alpha_{l^{-1}}(1_{l}1_{lq_{j}}1_{f})\delta_{l^{-1}f}
     \\
     &=\sum_{j=1}^{n}\lambda_j\alpha_{l^{-1}}(0)\delta_{l^{-1}f}=0,
    \end{align*}
    where we note that $1_{l}1_{lq_{j}}1_{f}=0$ for all $j\in \{1,\cdots, n\}$ since $X_{f}\cap X_{l} =\emptyset$.
    
    The above implies that $u^{-1}u(1_{f}\delta_{f})=u^{-1}0=0=1_{f}\delta_{f}$ and hence $X_{ f}=\emptyset$, a contradiction.

{\bf Case 2}  Suppose that $E$ has a single edge $l$ together with a finite number of isolated vertices. We show that the graph $E$ cannot have sinks. Let $u$ be as in Case~1.  Observe that
    \begin{align*}
    u^{-1}u&=\left(\sum_{i=1}^{m}\beta_i1_{l}1_{p_{i}}\delta_{l}\right)\left(\sum_{j=1}^{n}\lambda_j1_{l^{-1}}1_{q_{j}}\delta_{l^{-1}}\right)\\ 
    &=\sum_{i=1}^{m}\sum_{j=1}^{n}\beta_i\lambda_j\alpha_{l}(\alpha_{l^{-1}}(1_{l}1_{p_{i}})1_{l^{-1}}1_{q_{j}})\delta_{e}\\
    &=\sum_{i=1}^{m}\sum_{j=1}^{n}\beta_i\lambda_j1_l1_{p_{i}}1_{lq_{j}}\delta_{e},
    \end{align*}
    which gives
    $$\sum_{i=1}^{m}\sum_{j=1}^{n}\beta_i\lambda_j1_l1_{p_{i}}1_{lq_{j}}\delta_{e}=\sum_{v\in E^{0}}1_{v}\delta_{e}=1_{D(X)\star_{\alpha}\mathbb{F}}.$$
    Hence, 
    $$\sum_{i=1}^{m}\sum_{j=1}^{n}\beta_i\lambda_j1_{l}1_{p_{i}}1_{lq_{j}}=\sum_{v\in E^{0}}1_{v}$$
    and taking $w\in E^{0}$  a sink with  $s(l)\neq w$ we get
    $$\sum_{i=1}^{m}\sum_{j=1}^{n}\beta_i\lambda_j1_{l}1_{p_{i}}1_{lq_{j}}(w)=0,$$

    but $\sum_{v\in E^{0}}1_{v}(w)=1_{w}(w)=1$, which gives a contradiction. From this we get that $E$ is a  graph with one edge and one vertex (note that $s(f)=r(f)$ otherwise $r(f)$ is a sink), that is,  $E$ is a loop.

Conversely, assume that the graph is a loop with $E^{1}=\{f\}$ and $E^{0}=\{v\}$. Note that $X=\{f^\infty\}$ and $X_{p}=X_{v}=X_{p^{-1}}$ for all $p\in \mathbb{F}$, which implies that $1_{p}=1_{v}=1_{p^{-1}}$ for all $p\in \mathbb{F}$. Furthermore, $1_{v}\delta_{e}$ is the identity of $D(X)\star_{\alpha}\mathbb{F}$. With this in mind, notice that for each $p\in \mathbb{F}$ and $\lambda\in \K$ non-zero, we have that
$(\lambda 1_{p}\delta_{p})(\alpha_{p\m}(\lambda\m1_{p})\delta_{p^{-1}})(\lambda1_{p}\delta_{p})=1_{p}\delta_{p}.$ 
Hence, $D(X)\star_{\alpha}\mathbb{F}$ is graded unit-regular.
\end{proof}

 Combining Propositions \ref{loop}, \ref{loop1} and \ref{loop2}, we obtain our main
 theorem.

\begin{thm} \label{loop3} Let $E$ be a graph with at least one edge. Consider the Leavitt path algebra $L_{\K}(E)$ endowed with the $\mathbb{F}$-grading induced by the isomorphism \eqref{isop}. The following assertions are equivalent.
\begin{itemize}
\item  $L_{\K}(E)$ is strongly graded;
\item  $L_{\K}(E)$ is  graded clean;
\item $L_{\K}(E)$ is graded unit-regular;
\item $E$ is a loop.
\end{itemize}
\end{thm}
 Since for graphs with a single edge, the free group on the edges coincides with the integers we get, from  Theorem \ref{loop3} and \cite[Proposition 1.3.4]{AA}, the following result.
\begin{cor} There is a  $\Z$-grading in  the Laurent polynomial $\K$-algebra $\K[x,x\m]$, for which the following assertions are equivalent.
    \begin{itemize}
\item  $\K[x,x\m]$ is strongly graded;
\item  $\K[x,x\m]$ is  graded clean;
\item $\K[x,x\m]$ is graded unit-regular.
\end{itemize}\end{cor}
\begin{rem} Regarding Theorem \ref{loop3}, in \cite[Theorem 5.9]{PMS} it is shown that, for $n\in \Z^+$, there is a group $G$ such that the Leavitt path algebra associated with the graph \begin{center}
		$E_n:= \xymatrix{\bullet_{v_1} \ar@(ur,dr)^-{f_1}& \bullet_{v_2}\ar@(ur,dr)^-{f_2} & \bullet_{v_3}\ar@(ur,dr)^-{f_3} & \cdots &\bullet_{v_n}\ar@(ur,dr)^-{f_n} & \bullet_{v_{n+1}},
		}$  
  \end{center} 
  endowed with a certain canonical $G$-grading, is a sum of strongly  $G$-graded  rings and a ring with a trivial $G$-grading.  \footnote{Recall that a $G$-grading on a ring $R=\bigoplus_{g\in G}R_g$ is trivial if $R_e=R$ and $R_g=\{0\},$ for every $ g\neq e$.}
\end{rem}
\section{On the isomorphism between $D(X)\star_{\alpha}\mathbb{F}$ and $L_{\K}(E)$ }
\noindent We finish this work proposing a slight modification on the partial skew ring $D(X)\star_{\alpha}\mathbb{F}$ constructed in \cite{GR}, but still obtaining the algebra $L_{\K}(E)$ (in the case of graphs without isolated points). In \cite{GR}, the definition of the algebra $D(X)$ requires to endogenously add the characteristic functions $1_{v}$, for $v\in E^{0}$, to $D(X)$. In this section, we describe a construction where such maps are replaced with formal sums. For the reader's convenience, we recall the notion of formal series below.
\begin{defi}
A formal series on a $\K$-algebra $A$ is an infinite sum considered independent of any notion of convergence and can be manipulated with the usual algebraic operations on series. More precisely, the formal series satisfy the following rules:
\begin{itemize}
    \item $\displaystyle\sum_{i=1}^{\infty}a_{i}+\displaystyle\sum_{i=1}^{\infty}b_{i}=\displaystyle\sum_{i=1}^{\infty}(a_{i}+b_{i})$;
    \item $k\cdot\displaystyle\sum_{i=1}^{\infty}a_{i}= \displaystyle\sum_{i=1}^{\infty}k\cdot a_{i}$ for all $k\in \K$;
    \item $\left(\displaystyle\sum_{i=1}^{\infty}a_{i}\right)\left(\displaystyle\sum_{i=1}^{\infty}b_{i}\right)=\left(\displaystyle\sum_{i=1}^{\infty}c_{i}\right)$, where $c_{i}=\displaystyle\sum_{i=m+n}a_{m}a_{n}$.
\end{itemize}
\end{defi}
\noindent Following what is done in \cite{GR} (and recalled in Section~\ref{LPASR}), we keep the notations used to define the partial action $(X_{p}, \theta_{p})_{p\in \mathbb{F}}$ at the set level, and its induced partial action $(F(X_{p}), \alpha_{p})_{p\in \mathbb{F}}$, on the algebra level (where $F(X_p)\subseteq F(X)$ is the $\K$-algebra of functions from $X_p$ to $\K$).

Let $E$ be a graph without isolated vertices and let $v\in E^{0}$. If $v$ is an infinite emitter we consider the formal series $\displaystyle\sum_{i=1}^{\infty}1_{{(f_{i})}_{v}}, $ where $\{(f_{ 1})_{v},(f_{2})_{v},\cdots\}$ is the set of all edges such that $s(f)=v$. Define
$$D^s(X)=span_{\K}\left\{\{1_{p}\mid p\in \mathbb{F}\setminus \{0\}\}\cup \left\{\sum_{i=1}^{\infty}1_{{(f_{i})}_{v}}: v \text{ is an infinite emitter} \right\}\right\},$$ where $span_{\K}$ means the $\K$-linear span. Notice that each element of $D^s(X)$ can be seen as a formal sum, but the infinite formal sums only appear in the presence of infinite emitters. 

To define a partial action we will use \cite[Proposition~4.10]{Exel} and follow similar ideas to the ones presented in \cite[Section~5.1]{subshift}. For this, in the language of \cite[Section~4]{Exel}, we need to define a partial symmetry for each edge of the graph. In our case, such a partial symmetry is a bijection between subsets of $D^s(X)$, which we define in the following way:
For each $f\in E^1$, let $D^s_f=1_f D^s(X)$,  $D^s_{f^{-1}}=1_{f^{-1}} D^s(X)$ and let $\beta_f:D^s_{f^{-1}} \rightarrow D^s_f$ be the bijection obtained by defining, for $h\in \mathbb F$, $\beta_f(1_{f^{-1}} 1_h) = 1_f 1_{fh}$ and extending it linearly (including to infinite formal sums). Notice that $D^s_f$ and $D^s_{f^{-1}}$ are ideals of $D^s(X)$ for each $f\in E^1$. Furthermore, $\beta_f$ and $\beta_{f^{-1}}$  are ring homomorphisms, for any $f\in E^1$. Indeed, to verify this last statement, it is enough to check that
$$\beta_f((1_{f^{-1}} 1_h)( 1_{f^{-1}} 1_t)) =\beta_f(1_{f^{-1}} 1_h) \beta_f(1_{f^{-1}} 1_t),$$ for every $h,t\in \mathbb F$, which follows from the fact that the product of two functions $1_p$ and $1_r$, with $p,r\in \mathbb F$, is either equal to zero or equal to one of then.
Now, applying \cite[Proposition~4.10]{Exel} to the family $\{\beta_f\}_{f\in E^1}$, we obtain a unique set-theoretical partial action of the free group $\F$ on the edges on the set $D^s(X)$, which we denote by $(D^s_t,\beta_t)_{t\in \mathbb F}.$ It also follows from \cite[Proposition~4.10]{Exel} that $(D^s_t,\beta_t)_{t\in \mathbb F}$    is semi-saturated, in the sense that $\beta_{ts}=\beta_t \beta_s$ for all $s,t\in \mathbb F$ such that $|s+t|=|s|+|t|,$ where $|g|$ denotes the length of the word $g$ in $\F$. The partial action $(D^s_t,\beta_t)_{t\in \mathbb F}$ is also a partial action on the ring level because the maps $\beta_t, t\in \F$, are homomorphisms between ideals of $D^s(X)$ (since the generating partial symmetries are homomorphisms).  Moreover, for all $a,b \in E^1,$ the ideals $D^s_a$ and $D^s_b$ are orthogonal, that is $D^s_a \cap D^s_b=\{0\}.$ To finish the characterization  of  $(D^s_t,\beta_t)_{t\in \mathbb F}$, observe that if $a,b$  are  paths such that $r(a)=r(b)$ then, from the definition of $\beta$ and the fact that $(D^s_t,\beta_t)_{t\in \mathbb F}$ is semi-saturated, we have that $D^s_{ab^{-1}} =  1_{ab^{-1}}D^s(X)$ and, from orthogonality, 
we get that $D^s_t=\emptyset$ if $t\neq ab^{-1}$ for some paths $a,b$ such that $r(a)=r(b)$.



Next, we will construct an isomorphism between $L_{\K}(E)$ and $D^s(X)\star_{\beta}\mathbb{F}$, and for this we need the following remark.

\begin{rem}\label{x's}
Recall that if $v=r(f)$, then $X_{v}=X_{f^{-1}}$ and so $1_{v}=1_{f^{-1}}$.
\end{rem}

The desired isomorphism between $L_{\K}(E)$ and $D^s(X)\star_{\beta}\mathbb{F}$ is obtained following closely the steps done in the proof of \cite[Proposition 3.2.]{GR}. Therefore, we just highlight the main differences. The key point is that we need to find substitutes for the functions $1_{v}$ used in \cite[Proposition 3.2.]{GR}. Let us analyze the possible cases:
\begin{itemize}
    \item If $v$ is a sink: Since the graph has no isolated vertices, then there exists $f\in E^{1}$ such that $r(f)=v$. So, we can choose $f_{v}\in \{g\in E^{1}\mid r(g)=v\}$ and thus replace the function $1_{v}$ by $1_{f_{v}^{-1}}$ (note that $1_{v}=1_{f_{v}^{-1}}$ by Remark \ref{x's}).
    
    \item If $v$ is a regular vertex: In this case, Item~(4) in Definition~\ref{lpa} suggests that we  replace $1_{v}$ by the  sum $\displaystyle\sum_{i=1}^{n}1_{{(f_{i})}_{v}},$ where $\{f\in E^{1}\mid s(f)=v\}=\{{(f_{1})_{v}},\ldots,{(f_ {n})_{v}}\}$. 
    \item If $v$ is an infinite emitter: We replace $1_{v}$ by the formal series $\displaystyle\sum_{i=1}^{\infty}1_{{(f_{i})}_{v}} $, since $\{f_{i}\in E^{1}\mid s(f_{i})=v\}=X_{v}=\displaystyle\bigcup_{i=1}^{\infty}X_{{(f_{i})}_{v}}$.
\end{itemize}

 With the above, one can show that the sets $\{1_{f}\mid f \in E^{1}\}$, $\{1_{f^{-1}}\mid f\in E^{1}\}$ and $$\{1_{f_{v}^{-1}}\delta_{f}\mid v\text{ is a sink} \}\cup \left\{\displaystyle\sum_{i=1}^{n_v}1_{{(f_{i})}_{v}}\mid v \text{ is a regular vertex }\right\} \cup$$   $$\left\{\displaystyle\sum_{i=1}^{\infty}1_{{(f_{i})}_{v}}\mid v \text{ is an infinite emitter}\right\},$$ satisfy the relations that define the Leavitt path algebra. Using the universal property of $L_{\K}(E)$ (see \cite[Remark 1.2.5]{AA}), we obtain the desired homomorphism from $L_{\K}(E)$ to $D^s(X)\star_{\beta}\mathbb{F}$. Following analogously to the proof of \cite[Proposition 3.2.]{GR}, we conclude that this homomorphism is an isomorphism and we can now state the following theorem.

 \begin{thm}
     Let $E$ be a graph without isolated vertices. Then, $L_{\K}(E)$ and $D^s(X)\star_{\beta}\mathbb{F}$ are isomorphic.
 \end{thm}

\end{document}